\documentclass[12pt]{amsart}
\usepackage{amsmath}
\usepackage{amssymb}
\usepackage[all]{xy}
\usepackage{longtable}
\allowdisplaybreaks[1]

\newtheorem{theorem}{Theorem}[section]
\newtheorem{lemma}[theorem]{Lemma}

\newtheorem{corollary}[theorem]{Corollary}

\theoremstyle{definition}

\theoremstyle{remark}

\numberwithin{equation}{section}
\DeclareMathOperator{\m}{m}
\DeclareMathAlphabet{\matheur}{U}{eur}{m}{n}

\newmuskip\pFqskip
\mathchardef\pFcomma=\mathcode`, 

\newcommand*\pFq[5]{%
  \begingroup
  \begingroup\lccode`~=`,
    \lowercase{\endgroup\def~}{\pFcomma\mkern\pFqskip}%
  \mathcode`,=\string"8000
  {}_{#1}F_{#2}\biggl(\genfrac{}{}{0pt}{}{#3}{#4};#5\biggr)%
  \endgroup
}

\renewcommand{\d}{{\mathrm d}}

\renewcommand{\Re}{\operatorname{Re}}

\begin{document}

\title[The Mahler measure of a Calabi-Yau threefold]{The Mahler measure of a Calabi-Yau threefold \\ and special $L$-values}

\author{Matthew A. Papanikolas}
\address{Department of Mathematics, Texas A{\&}M University, College
Station,
TX 77843, USA} \email{map@math.tamu.edu}

\author{Mathew D. Rogers}
\address{Department of Mathematics and Statistics, Universit\'{e} de Montr\'{e}al, CP 6128 succ. Centre-ville, Montr\'{e}al Qu\'{e}bec H3C 3J7, Canada} \email{mathewrogers@gmail.com}

\author{Detchat Samart}
\address{Department of Mathematics, Texas A{\&}M University, College
Station,
TX 77843, USA} \email{detchats@math.tamu.edu}

\thanks{The first and the third authors were partially supported by NSF Grant DMS-1200577}


\date{May 9, 2013}

\begin{abstract}
The aim of this paper is to prove a Mahler measure formula of a four-variable Laurent polynomial whose zero locus defines a Calabi-Yau threefold. We show that its Mahler measure is a rational linear combination of a special $L$-value of the normalized newform in $S_4(\Gamma_0(8))$ and a Riemann zeta value. This is equivalent to a new formula for a $_6F_5$-hypergeometric series evaluated at $1$.
\end{abstract}


\maketitle
\section{Introduction}\label{Sec:intro}
For a nonzero $n$-variable Laurent polynomial $P$ with complex coefficients, the Mahler measure of $P$ is defined by $$\m(P)=\int_0^1\cdots \int_0^1 \log |P(e^{2\pi i \theta_1},\ldots,e^{2\pi i \theta_n})|\,d\theta_1\cdots d\theta_n.$$
If $P$ is a monic polynomial in one variable, then it follows by Jensen's formula that
$$\m(P)=\sum_{j=1}^n \max\{0,\log{|\alpha_j|}\},$$
where $\alpha_j$'s are the roots of $P$, and thus if $P$ is the irreducible polynomial of an algebraic number $\alpha$ over $\mathbb{Q}$, then $\m(P)$ is the logarithmic Weil height of $\alpha$.

When $P$ has more than one variable, however, there is no general closed form for $\m(P)$, and its explicit value is usually difficult to compute. On the other hand, in some particular cases when $P$ is a two-variable polynomial with rational coefficients whose zero set define a genus-one curve $C$, it turns out that $\m(P)$ is a rational multiple of $L'(E,0)$, where $E$ is the elliptic curve arising from $C$. The first known example of polynomials having this property is the family
\begin{equation*}
P_k:=x+x^{-1}+y+y^{-1}-k,
\end{equation*}
where $k\in\mathbb{Z}\backslash\{0,\pm4\}.$
This phenomenon was first observed by Deninger \cite{Deninger} and has been studied extensively by Boyd, Rodriguez Villegas, and many others \cite{Boyd,LR,RV,RZ}. Note that, unlike the other values of $k\in\mathbb{N}$, $P_4=0$ define a curve of genus $0$, and it was shown in \cite{RV} that $\m(P_4)=4G/\pi$, where $G=L(\chi_{-4},2)$, also known as Catalan's constant. Applying the definition of the Mahler measure directly, it is easy to show that
\begin{equation*}
\m((x+x^{-1})(y+y^{-1})-k)= \m(P_k)
\end{equation*}
for every $k$, so we have
\begin{equation}\label{E:2var}
\m((x+x^{-1})(y+y^{-1})-4)=\frac{4G}{\pi}=2L'(\chi_{-4},-1).
\end{equation}
As a higher dimensional analogue of the family $P_k$, Bertin and others \cite{BertinI,BertinIII,BertinII,BFFLM} studied the three-variable polynomials
$$Q_k:=x+x^{-1}+y+y^{-1}+z+z^{-1}-k,$$
whose zero loci define $K3$ surfaces $X_k$ over $\mathbb{Q}$. They proved that, for certain values of $k$ defining singular $K3$ surfaces, their Mahler measures are of the form
\begin{equation*}
\m(Q_k)=c_1L'(g,0)+c_2L'(\chi,-1),
\end{equation*}
where $c_1,c_2\in\mathbb{Q}$, $g$ is the weight $3$ newform associated with $X_k$, and $\chi$ is a quadratic character.  Afterwards,  the second and third authors \cite{Rogers,Samart2,Samart1} established similar results for other families of three-variable polynomials, including the formula
\begin{equation}\label{E:3var}
\m((x+x^{-1})(y+y^{-1})(z+z^{-1})-8)=4L'(h,0),
\end{equation}
where $h(\tau)=\eta^6(4\tau)$, and $\eta(\tau)$ is the Dedekind eta function. Therefore, it is natural to consider whether a four-variable analogue of \eqref{E:2var} and \eqref{E:3var}
can be expressed in terms of some special $L$-values.  Namely if we let
\begin{equation}
R_k := (x+x^{-1})(y+y^{-1})(z+z^{-1})(w+w^{-1})-k,
\end{equation}
then the main purpose of this paper is to prove the following result about $\m(R_{16})$.
\begin{theorem}\label{T:main}
The following identity is true:
\begin{equation*}
\m(R_{16}) = \m((x+x^{-1})(y+y^{-1})(z+z^{-1})(w+w^{-1})-16)=8L'(f,0)-28\zeta'(-2),
\end{equation*}
where $f(\tau)=\eta^4(2\tau)\eta^4(4\tau)$ is the unique normalized newform in $S_4(\Gamma_0(8))$, and $\zeta(s)$ denotes the Riemann zeta function.
\end{theorem}
Observe that, by the well-known functional equations:
\begin{align*}
\left(\frac{\sqrt{8}}{2\pi}\right)^s\Gamma(s)L(f,s)&=\left(\frac{\sqrt{8}}{2\pi}\right)^{4-s}\Gamma(4-s)L(f,4-s),\\
\zeta(s)&=2^s\pi^{s-1}\sin\left(\frac{\pi s}{2}\right)\Gamma(1-s)\zeta(1-s),
\end{align*}
the formula in Theorem~\ref{T:main} can be rephrased as
\begin{equation}\label{E:main}
\m(R_{16}) = \m((x+x^{-1})(y+y^{-1})(z+z^{-1})(w+w^{-1})-16)=\frac{192}{\pi^4}L(f,4)+\frac{7\zeta(3)}{\pi^2}.
\end{equation}
To prove \eqref{E:main}, we require some new formulas for the $L$-value and $\zeta(3)/\pi^2$. These formulas will be verified in the subsequent sections. It is worth pointing out that the Wilf-Zeilberger method plays an important role in simplifying parts of the proofs involving difficult integrals. Another example of the WZ method applied to proving relations between Mahler measures can be found in \cite{WZ}.  Whereas most of the proved formulas in the lower dimensional cases involve CM newforms, we note also that $f$ is a non-CM newform.  This is a major reason why we require new techniques in the proof that differ from the CM case.

Let us conclude this section by stating a crucial result relating Mahler measures to hypergeometric series, which will be used later.  Combined with \eqref{E:main} this theorem also implies the hypergeometric evaluation,
\begin{equation}
  \pFq{6}{5}{\frac{3}{2},\frac{3}{2},
\frac{3}{2},\frac{3}{2},1,1}{2,2,2,2,2}{1} = 128\log 2 - \frac{6144}{\pi^4} L(f,4) - \frac{224}{\pi^2} \zeta(3).
\end{equation}
The following general result can be proved easily using standard techniques from the theory of Mahler measures. (See for example \cite[Prop. 2.2]{Rogers}.)

\begin{theorem}\label{T:hypergeom}
If $|k|\geq 16$, then
\begin{equation*}
\m(R_k)=\Re\left(\log(k)-\frac{8}{k^2}\pFq{6}{5}{\frac{3}{2},\frac{3}{2},
\frac{3}{2},\frac{3}{2},1,1}{2,2,2,2,2}{\frac{256}{k^2}}
\right),
\end{equation*}
where
\begin{equation*}
\pFq{p}{q}{a_1,a_2,\ldots,a_p}{b_1,b_2,\ldots,b_q}{x}=\sum_{n=0}^{\infty}\frac{(a_1)_n\cdots(a_p)_n}{(b_1)_n\cdots(b_q)_n}\frac{x^n}{n!},
\end{equation*}
and $(c)_n=\Gamma(c+n)/\Gamma(c)$.
\end{theorem}

\section{A formula for $L(f,4)$} \label{Sec:Lseries}
Throughout this paper we will use the notation
\begin{equation}
F(\alpha):=\pFq{2}{1}{\frac{1}{2},\frac{1}{2}}{1}{\alpha},
\end{equation}
and we also employ the classical notations for elliptic integrals:
\begin{equation}
K(k)=\frac{\pi}{2}\pFq{2}{1}{\frac{1}{2},\frac{1}{2}}{1}{k^2},
\quad K'(k)=\frac{\pi}{2}\pFq{2}{1}{\frac{1}{2},\frac{1}{2}}{1}{1-k^2}.
\end{equation}
We first give some integral representations for $L(f,4)$ and Ap\'{e}ry's constant $\zeta(3)$ in terms of the product $KK'$. The main idea of the proof below essentially comes from \cite{RZ0}.
\begin{theorem} The following formulas are true:
\begin{align}
\frac{192}{\pi}L(f,4)=&-8\int_{0}^{1} \left(\frac{1+k^2}{1-k^2}\right) K(k) K'(k)\log k\,\d k,\label{L(f,4) intermed}\\
7\pi\zeta(3)=&-8\int_{0}^{1} \left(\frac{2k}{1-k^2} \right)K(k)K'(k)\log k\, \d k\label{zeta(3) intermed}
\end{align}
\end{theorem}
\begin{proof}
Formula \eqref{zeta(3) intermed} is a trivial rearrangement of the following identity:
\begin{equation}\label{E:Wan}
\frac{7}{8}\pi\zeta(3)=\int_{0}^{1}\left(\frac{-\log(1-k^2)}{k}\right) K(k) K'(k)\,\d k,
\end{equation}
due to Wan \cite{Wan}.
For any positive real number $u$, let $q=q(u):=e^{-2\pi u}.$ Then
\begin{align*}
f(i u)=&\left(\frac{\eta^2(4i u)}{\eta(2 i u)}\right)^4\left(\frac{\eta^2(2i u)}{\eta(4i u)}\right)^4
= q\psi^4(q^2)\varphi^4(-q^2),
\end{align*}
where
\begin{equation*}
\psi(q)=\sum_{n=0}^{\infty}q^{n(n+1)/2},\quad \varphi(q)=\sum_{n=-\infty}^{\infty}q^{n^2}.
\end{equation*}
Then
\begin{equation*}
\eta^4(2 i u)\eta^4(4i u)=\frac{1}{u^2}~e^{-2\pi u}\psi^4(e^{-4\pi u}) e^{-\frac{\pi}{4u}}\psi^4\left(e^{-\frac{\pi}{2u}}\right).
\end{equation*}
Applying a result from Ramanujan's notebooks \cite{Berndt},
\begin{equation*}
q\psi^4(q^2)=\sum_{n,k\ge 0}(2n+1)q^{(2n+1)(2k+1)},
\end{equation*}
we can write the cusp form as a four-dimensional series:
\begin{equation*}
\eta^4(2 i u)\eta^4(4i u)=\frac{1}{u^2}~\sum_{n,k,j,r\ge 0}(2n+1)(2j+1)e^{-2\pi u (2n+1)(2k+1)-\frac{\pi}{4 u}(2j+1)(2r+1)}.
\end{equation*}
Taking the Mellin transform both of sides yields
\begin{equation*}
\frac{6 L(f,4)}{(2\pi)^4}=\sum_{n,k,j,r\ge 0}(2n+1)(2j+1)\int_{0}^{\infty}u e^{-2\pi u (2n+1)(2k+1)-\frac{\pi}{4 u}(2j+1)(2r+1)}\,\d u.
\end{equation*}
Next, we use the transformation $u\mapsto (2j+1) u/(2n+1)$ to obtain:
\begin{align*}
\frac{6L(f,4)}{(2\pi)^4}&= \int_{0}^{\infty}u\sum_{j,k\ge 0}(2j+1)^3 e^{-2\pi u (2j+1)(2k+1)}\sum_{n,r\ge 0}\frac{1}{2n+1}e^{-\frac{\pi}{4 u}(2n+1)(2r+1)}\,\d u\\
&=\frac{1}{2}\int_{0}^{\infty}u\sum_{j\ge 0} \frac{(2j+1)^3 e^{-2\pi u (2j+1)}}{1-e^{-4\pi u (2j+1)}}\log\prod_{r=1}^{\infty}\frac{\left(1-e^{-\frac{\pi r}{2 u}}\right)^3}{\left(1-e^{-\frac{\pi r}{ u}}\right)\left(1-e^{-\frac{\pi r}{4 u}}\right)^2}\,\d u\\
&=\frac{1}{2}\int_{0}^{\infty}u\sum_{j\ge 0} \frac{(2j+1)^3 e^{-2\pi u (2j+1)}}{1-e^{-4\pi u (2j+1)}}\log\left(\frac{1}{\sqrt{2}}\frac{\eta^3(4 i u)}{\eta(2 i u)\eta^2(8 i u)}\right)\,\d u.
\end{align*}
Let $u=\frac{1}{4}\frac{F(1-\alpha)}{F(\alpha)}$. Then the new region of integration is $\alpha\in[1,0]$. We also have the formulas
\begin{align*}
\frac{1}{\sqrt{2}}\frac{\eta^3(4 i u)}{\eta(2 i u)\eta^2(8 i u)}&=\alpha^{-1/8},\\
\sum_{j\ge 0} \frac{(2j+1)^3 e^{-2\pi u (2j+1)}}{1-e^{-4\pi u (2j+1)}}&=\frac{1}{4} \sqrt{\alpha} (1 + \alpha) F^4(\alpha),\\
\frac{\d u}{\d\alpha}&=-\frac{1}{4\pi\alpha(1-\alpha)F^2(\alpha)}.
\end{align*}
Thus
\begin{equation*}
\frac{6 L(f,4)}{(2\pi)^4}=-\frac{1}{1024\pi}\int_{0}^{1}\frac{(1+\alpha)}{\sqrt{\alpha}(1-\alpha)}F(1-\alpha)F(\alpha)\log \alpha\,\d \alpha.
\end{equation*}
Finally, formula \eqref{L(f,4) intermed} follows by setting  $\alpha=k^2$.
\end{proof}
\begin{corollary}
The following formulas are true:
\begin{align}
\frac{192}{\pi^4}L(f,4)+\frac{7\zeta(3)}{\pi^2}&=\frac{8}{\pi^3}\int_{0}^{1}K(k)K'(k)\log\left(\frac{1+k}{1-k}\right)\frac{\d k}{k}\label{corollary formula},\\
\frac{12}{\pi}L(f,4)&=\int_{0}^{1}K(k)K'(k)\log(1+k)\frac{\d k}{k} \label{cor formula2},\\
-\frac{12}{\pi}L(f,4)-\frac{7}{8}\pi\zeta(3)&=\int_{0}^{1}K(k)K'(k)\log(1-k)\frac{\d k}{k} \label{cor formula3}.
\end{align}
\end{corollary}
\begin{proof}
Add formulas \eqref{L(f,4) intermed} and \eqref{zeta(3) intermed}. Then formula \eqref{corollary formula} follows immediately by applying the transformation $k\mapsto (1-k)/(1+k)$, and the identities
\begin{equation*}
K\left(\frac{1-k}{1+k}\right)=\frac{1+k}{2}K'(k), \quad K'\left(\frac{1-k}{1+k}\right)=(1+k)K(k).
\end{equation*}
Formulas \eqref{cor formula2} and \eqref{cor formula3} are merely trivial consequences of \eqref{corollary formula} and \eqref{E:Wan}.
\end{proof}

Now we proceed to the most difficult part of the calculation.  The following lemma is derived by the Wilf-Zeilberger method.

\begin{lemma}\label{L:WZ} The following identities are true when $n\ge 0$:
\begin{align}
\sum_{k=0}^{n}\frac{1}{2^{4k}}\binom{2k}{k}^2\frac{1}{2n-2k+1}&=\sum_{k=0}^{n}\frac{1}{2^{4k}}\binom{2k}{k}^2\frac{1}{n+k+1}\label{WZ crucial}\\
&=\frac{2^{4n}}{(2n+1)^2\binom{2n}{n}^2}\sum_{k=0}^{n}\frac{(4k+1)}{2^{8k}}
\binom{2k}{k}^4.\label{WZ noncrucial}
\end{align}
\end{lemma}

\begin{proof} We say that $f(n,k)$ and $g(n,k)$ are a WZ pair if they satisfy the relation:
\begin{equation*}
f(n+1,k)-f(n,k)=g(n,k+1)-g(n,k).
\end{equation*}
Since the function $g$ telescopes, summing both sides from $k=0$ to $k=n$ and adding $f(n+1,n+1)$ to either side yields
\begin{equation*}
\sum_{k=0}^{n+1}f(n+1,k)-\sum_{k=0}^{n}f(n,k)=f(n+1,n+1)+g(n,n+1)-g(n,0).
\end{equation*}
Thus, if we let
\begin{equation*}
h(n):=\sum_{k=0}^{n}f(n,k),
\end{equation*}
then the relation above is equivalent to
\begin{equation*}
h(n+1)-h(n)=f(n+1,n+1)+f(n,n+1)-g(n,0).
\end{equation*}
Finally, iterate down to zero, and let $n\mapsto n-1$, to obtain
\begin{equation}\label{WZ intermed}
h(n)=h(0)+\sum_{j=1}^{n} \bigl(f(j,j)+g(j-1,j)-g(j-1,0)\bigr).
\end{equation}
Now we substitute the following WZ pairs into \eqref{WZ intermed}, to recover equations \eqref{WZ crucial} and \eqref{WZ noncrucial}:
\begin{align*}
f(n,k)&=\frac{1}{2^{4 k+4 n}}\frac{ (2 n+1)^2 }{2n-2 k+1}\binom{2 k}{k}^2 \binom{2 n}{n}^2,\\
g(n,k)&=-\frac{1}{2^{4 k+4 n}}\frac{ k^2 (2 n+1)^2 }{(1+n)^2 (2n-2 k+3)} \binom{2 k}{k}^2 \binom{2 n}{n}^2,
\end{align*}
and also
\begin{align*}
f(n,k)&=\frac{1}{2^{4 k+4 n}}\frac{ (2 n+1)^2 }{n+ k+1}\binom{2 k}{k}^2 \binom{2 n}{n}^2,\\
g(n,k)&=\frac{1}{2^{4 k+4 n}}\frac{ k^2 (2 n+1)^2 }{(n+1)^2 (n+ k+1)}\binom{2 k}{k}^2 \binom{2 n}{n}^2.
\end{align*}
\end{proof}

\begin{theorem}The following formulas are true:
\begin{equation}
\frac{8}{\pi^3}\int_{0}^{1}K(k)K'(k)\log\left(\frac{1+k}{1-k}\right)\frac{\d k}{k}=2\sum_{n=0}^{\infty}\frac{1}{(2n+1)^2}\sum_{k=0}^{n}\frac{(4k+1)}{2^{8k}} \binom{2k}{k}^4,\label{elliptical integral double sum}
\end{equation}
\begin{equation}
\frac{14\zeta(3)}{\pi^2}+\sum_{n=0}^{\infty}\frac{1}{2^{8n}} \binom{2n}{n}^4 \frac{1}{2n+1}=2\sum_{n=0}^{\infty}\frac{1}{(2n+1)^2}\sum_{k=0}^{n} \frac{(4k+1)}{2^{8k}}  \binom{2k}{k}^4.\label{elliptical integral double sum 2}
\end{equation}
\end{theorem}

\begin{proof} Expand the logarithm in a Taylor series, and then apply Wan's formula for the moments \cite{Wan}:
\begin{equation*}
\int_{0}^{1}k^mK(k)K'(k)\,\d k = \frac{\pi^2}{8}\frac{\Gamma\left(\frac{m+1}{2}\right)^2}{\Gamma\left(\frac{m+2}{2}\right)^2} \,\pFq{4}{3}{\frac{1}{2},\frac{1}{2},\frac{m+1}{2},\frac{m+1}{2}}{1,\frac{m+2}{2},\frac{m+2}{2}}{1},
\end{equation*}
to obtain
\begin{align*}
\frac{8}{\pi^3}\int_{0}^{1}K(k)K'(k)\log\left(\frac{1+k}{1-k}\right)\frac{\d k}{k}&=2\sum_{m=0}^{\infty}\frac{1}{2m+1}\frac{\left(\frac1 2\right)_m^2}{(1)_m^2}\, \pFq{4}{3}{\frac{1}{2},\frac{1}{2},m+\frac{1}{2},m+\frac{1}{2}}{1,m+1,m+1}{1}\\
&=2\sum_{m=0}^{\infty}\frac{1}{2m+1}\frac{\left(\frac1 2\right)_m^2}{(1)_m^2}\sum_{k=0}^{\infty}\frac{\left(\frac12\right)_k^2\left(\frac12+m\right)_k^2}{(1)_k^2(1+m)_k^2}\\
&=2\sum_{m=0}^{\infty}\frac{1}{2m+1}\sum_{k=0}^{\infty}\frac{\left(\frac12\right)_k^2\left(\frac12\right)_{m+k}^2}{(1)_k^2(1)_{m+k}^2}\\
&=2\sum_{n=0}^{\infty}\frac{\left(\frac12\right)_{n}^2}{(1)_{n}^2}\sum_{k=0}^{n}\frac{\left(\frac12\right)_k^2}{(1)_k^2}\frac{1}{2n-2k+1}.
\end{align*}
The last step follows from setting $n=m+k$.  Then we apply Lemma~\ref{L:WZ} to the inner summation to complete the proof of \eqref{elliptical integral double sum}.

To prove formula \eqref{elliptical integral double sum 2}, first recall that the square of any infinite series can be written as
\begin{equation}\label{trivial summation rearrangement}
\left(\sum_{n=0}^{\infty}a(n)\right)^2+\sum_{n=0}^{\infty}a(n)^2=2\sum_{n=0}^{\infty}a(n)\sum_{j=0}^{n}a(j).
\end{equation}
Since the elliptical integral $K(k)$ can be expressed as
\begin{equation*}
\frac{2}{\pi}K(k)=\sum_{n=0}^{\infty}\frac{1}{2^{4n}} \binom{2n}{n}^2 k^{2n},
\end{equation*}
formula \eqref{trivial summation rearrangement} becomes
\begin{equation*}
\frac{4}{\pi^2}k K^2(k)+\sum_{n=0}^{\infty}\frac{1}{2^{8n}} \binom{2n}{n}^4 k^{4n+1}=2\sum_{n=0}^{\infty}\frac{1}{2^{4n}} \binom{2n}{n}^2\sum_{j=0}^{n}\frac{1}{2^{4j}} \binom{2j}{j}^2 k^{2n+2j+1}.
\end{equation*}
Next, integrate both sides for $k\in[0,1]$.  We then appeal to Wan's result
\begin{equation*}
\frac{4}{\pi^2}\int_{0}^{1}k K^2(k)\,\d k=\frac{7\zeta(3)}{\pi^2},
\end{equation*}
and simplify the resulting double-sum on the right with \eqref{WZ noncrucial}.  This completes the proof of \eqref{elliptical integral double sum 2}.
\end{proof}
To complete the proof of Theorem~\ref{T:main}, we need an additional formula for $\zeta(3)/\pi^2$, which will be proved in the next section.

\section{A formula for $\zeta(3)/\pi^2$ and proof of the main theorem} \label{Sec:zeta}
The goal of this section is to evaluate a certain hypergeometric series in terms of $\zeta(3)/\pi^2$ and $\log2$.  Equation \eqref{zeta3 formula} is the last ingredient that we need to complete the proof of our main result, Theorem \ref{T:main}.  While this formula is essentially a `byproduct' of the Mahler measure considerations, it is actually interesting in its own right.  It is quite significant that the hypergeometric series has four binomial coefficients in the numerator.  It turns out that there are very few instances where hypergeometric functions with more than three binomial coefficients have been explicitly evaluated.  Ramanujan proved many formulas for cases with three binomial coefficients; those results are closely tied to questions about modular forms and class numbers \cite{Borwein, Chudnovsky, Ra}.  More recently, Guillera discovered many conjectural formulas for constants like $1/\pi^2$ using numerical searches.  Most of Guillera's identities involve four or more binomial coefficients, and only a select few of his results have been rigorously proved.  A full survey of recent developments is beyond the scope of this paper, but we refer the interested reader to \cite{Borwein 2,G10Rama,ZS}.

\begin{theorem}\label{zeta3 theorem}The following formula is true:
\begin{equation}\label{zeta3 formula}
-\frac{14\zeta(3)}{\pi^2}+4\log2=1+\sum_{n=1}^{\infty}\frac{(4n+1)}{(2n)(2n+1)} \binom{2n}{n}^4\frac{1}{2^{8n}}.
\end{equation}
\end{theorem}
We need several auxiliary results to prove this theorem.  First, define the following Mahler measures:
\begin{align}
m(\alpha)&:= \m\bigl(\alpha+x+x^{-1}+y+y^{-1}\bigr),\\
R(\alpha)&:= \m\bigl(\alpha(u+u^{-1})(z+z^{-1})+(x+x^{-1})(y+y^{-1})\bigr).\label{R def}
\end{align}
Our strategy is to reduce $R(\alpha)$ to trilogarithms immediately.  Then we prove an integral representation for $R(\alpha)$ involving $m(\alpha)$ and an elliptical integral. Substituting Fourier series expansions then leads to
an expression for $R(1)$ in terms of hypergeometric functions.  We complete the proof of \eqref{zeta3 formula} by comparing the two different formulas for $R(1)$.

\begin{lemma} Suppose that $0\le \alpha\le1$.  The following formula is true:
\begin{align}
R(\alpha)=&\frac{4}{\pi^2}\sum_{n=0}^{\infty}\frac{\alpha^{2n+1}}{(2n+1)^3}.\label{R as polylog}
\end{align}
\end{lemma}
\begin{proof}
Setting $u=e^{2\pi i t}$ and $z=e^{2\pi i s}$ in the definition of $R(\alpha)$ above,  with a little work we find
\begin{equation}\label{R intermediate}
R(\alpha)=\int_{0}^{1}\int_{0}^{1}m\left(4\alpha|\cos(2\pi t)\cos(2\pi s)|\right)\, \d s\d t.
\end{equation}
Now we require the $_3F_2$ series expansion for $m(\alpha)$. If $\alpha\in[0,1]$, we have
\begin{equation*}
m(4\alpha)=4\sum_{n=0}^{\infty} \binom{2n}{n}^2\frac{\left(\alpha/4\right)^{2n+1}}{2n+1}.
\end{equation*}
Therefore, $R(\alpha)$ becomes
\begin{align*}
R(\alpha)&=4\sum_{n=0}^{\infty} \binom{2n}{n}^2 \frac{\left(\alpha/4\right)^{2n+1}}{2n+1} \int_{0}^{1}\int_{0}^{1}|\cos(2\pi s)\cos(2\pi t)|^{2n+1}\,\d t\d s\\
&=\frac{4}{\pi^2}\sum_{n=0}^{\infty}\frac{\alpha^{2n+1}}{(2n+1)^3},
\end{align*}
where the final step uses the formula
\begin{equation*}
\int_{0}^{1}|\cos(2\pi t)|^{2n+1}\d t=\frac{2^{2n+1}}{\pi (2n+1)\binom{2n}{n}}.
\end{equation*}
\end{proof}

\begin{lemma}Suppose that $F(k)$ is integrable for $k\in[0,1]$.  Then
\begin{equation}\label{density integral}
\int_{0}^{1}\int_{0}^{1}F\left(\left|\cos(2\pi t)\cos(2\pi s)\right|\right)\d s\d t=\frac{4}{\pi^2}\int_{0}^{1}F(k) K'(k)\,\d k.
\end{equation}
Thus for any $\alpha\in\mathbb{C}$:
\begin{align}
R(\alpha)=&\frac{4}{\pi^2}\int_{0}^{1}m\left(4\alpha k\right)K'(k)\,\d k.\label{R as integral}
\end{align}
\end{lemma}
\begin{proof}
To prove \eqref{density integral}, notice that by an elementary change of variables:
\begin{align*}
\int_{0}^{1}\int_{0}^{1}F\left(\left|\cos(2\pi t)\cos(2\pi s)\right|\right)\,\d s\d t&=\frac{4}{\pi^2}\int_{0}^{1}\int_{0}^{1}F\left(u v\right)\frac{\d u \d v}{\sqrt{(1-u^2)(1-v^2)}}\\
&=\frac{4}{\pi^2}\int_{0}^{1}F(k)\left(\int_{k}^{1}\frac{\d v}{v\sqrt{(1-\frac{k^2}{v^2})(1-v^2)}}\right)\,\d k.
\end{align*}
The second step follows from setting $k= u v$ and eliminating $u$.  We then use the change of variables $v\mapsto\sqrt{1-(1-k^2)t^2}$ to identify the nested integral as $K'(k)$. Finally, formula \eqref{R as integral} follows from applying \eqref{density integral} to \eqref{R intermediate}.
\end{proof}

\begin{lemma}  We have the following Fourier series expansions:
\begin{align}
K(\sin(\theta))\cos(\theta)&=\frac{\pi}{2}\sum_{n=0}^{\infty}\frac{1}{2^{4n}} \binom{2n}{n}^2 \left(\sin(4n\theta)+\sin((4n+2)\theta)\right),\label{Ksin Fourier}\\
K(\cos(\theta))\cos(\theta)&=\frac{\pi}{2}\sum_{n=0}^{\infty}\frac{1}{2^{4n}}  \binom{2n}{n}^2 \left(\cos(4n\theta)+\cos((4n+2)\theta)\right),\label{Kcos Fourier}\\
m\left(4\sin(\theta)\right)&=\log2-\sum_{n=1}^{\infty}\frac{1}{2^{4n}} \binom{2n}{n}^2 \frac{\cos(4n\theta)}{4n}-\sum_{n=0}^{\infty}\frac{1}{2^{4n}} \binom{2n}{n}^2 \frac{\cos((4n+2)\theta)}{4n+2}.\label{m Fourier}
\end{align}
\end{lemma}
\begin{proof} Formulas \eqref{Ksin Fourier} and \eqref{Kcos Fourier} follow immediately from results in \cite{Wan}.  The $_3F_2$ series expansion for $m(\alpha)$ is equivalent to
\begin{equation*}
m(4\sin(\theta))=\frac{2}{\pi}\int_{0}^{\sin(\theta)}K(k)\,\d k
=\frac{2}{\pi}\int_{0}^{\theta}K(\sin(u))\cos(u)\,\d u.
\end{equation*}
We then recover \eqref{m Fourier} by integrating \eqref{Ksin Fourier}.
\end{proof}

\begin{proof}[Proof of Theorem \ref{zeta3 theorem}]  Equation \eqref{R as polylog} immediately gives
\begin{equation*}
R(1)=\frac{7\zeta(3)}{2\pi^2}.
\end{equation*}
On the other hand, equation \eqref{R as integral} leads to
\begin{align*}
R(1)&=\frac{4}{\pi^2}\int_{0}^{\pi/2}m(4\sin(\theta))K'(\sin(\theta))\cos(\theta)\,\d\theta\\
&=\log2-\frac{1}{4}\sum_{n=1}^{\infty}\frac{1}{2n} \binom{2n}{n}^{4} \frac{1}{2^{8n}}-\frac{1}{4}\sum_{n=0}^{\infty}\frac{1}{2n+1} \binom{2n}{n}^4\frac{1}{2^{8n}}.
\end{align*}
The final step follows from substituting \eqref{m Fourier} and \eqref{Kcos Fourier} into the integral.  Comparing the values of $R(1)$ concludes the proof.
\end{proof}

\begin{proof}[Proof of Theorem~\ref{T:main}]
Combining formulas \eqref{corollary formula}, \eqref{elliptical integral double sum}, and \eqref{elliptical integral double sum 2} leads to the following identity:
\begin{equation}\label{result 1}
\frac{192}{\pi^4}L(f,4)-\frac{7\zeta(3)}{\pi^2}=\sum_{n=0}^{\infty}\frac{1}{2n+1} \binom{2n}{n}^4 \frac{1}{2^{8n}}.
\end{equation}
By Theorem~\ref{T:hypergeom}, we have
\begin{align*}
\m\left(R_{16}\right)&=\Re\left(\log(16)-\frac{1}{32}\pFq{6}{5}{\frac{3}{2},\frac{3}{2},\frac{3}{2},\frac{3}{2},1,1}{2,2,2,2,2}{1}\right)\\
&=4\log(2)-\sum_{n=1}^{\infty}\frac{1}{2n} \binom{2n}{n}^4\frac{1}{2^{8n}}\\
&=4\log(2)-\sum_{n=1}^{\infty}\frac{4n+1}{(2n)(2n+1)} \binom{2n}{n}^4 \frac{1}{2^{8n}}+\sum_{n=1}^{\infty}\frac{1}{2n+1}  \binom{2n}{n}^4\frac{1}{2^{8n}}.
\end{align*}
Finally, we obtain \eqref{E:main} using \eqref{zeta3 formula} and \eqref{result 1}.
\end{proof}

\section{Concluding remarks}\label{Sec:conclude}
With Theorem~\ref{T:main} in mind it is natural to relate $\m(R_k)$ to special values of $L$-functions for other values of $k$, though this remains a challenge.  Through work of Deninger~\cite{Deninger}, we expect the Mahler measure to encode information about the $L$-series of the zero locus of the polynomial.  If we define a family $H_t$ of hypersurfaces in four-dimensional affine space by
\[
  H_t : (x^2+1)(y^2+1)(z^2+1)(w^2+1) - 16txyzw = 0,
\]
then it is shown in \cite{FPR} that for any odd prime $p$ we can find a formula for the number of points on $H_t$ over $\mathbb{F}_p$ in terms of finite field hypergeometric functions.  In particular,
\begin{multline} \label{FPR}
  \# H_t(\mathbb{F}_p) = p^3{}_4F_3(t^2) + 4\phi(-1)p^2{}_2F_1(t^2) - 3\varepsilon(t^2-1)p^2 \\
  {}+ p^3 + 8(\phi(-1)+1)p^2 - 16(\phi(-1)+1)p - 3p + 8(\phi(-1)+1) + 1,
\end{multline}
where if $\phi$ is the Legendre symbol modulo $p$ and $\varepsilon$ is the trivial character modulo $p$ (taking value $0$ at $0$), then ${}_4F_3(x) = {}_4F_3(\phi,\phi,\phi,\phi; \varepsilon,\varepsilon, \varepsilon; x)$ and ${}_2F_1(x) = {}_2F_1(\phi,\phi; \varepsilon; x)$ are the finite field hypergeometric functions originally defined by Greene~\cite{Greene}.  Ahlgren and Ono showed in \cite{AO1,AO2} that
\[
  p^3{}_4F_3(1) = -a_p - p,
\]
where $f = \sum_{n=1}^\infty a_n q^n$ is the Fourier expansion of the modular form $f = \eta^4(2\tau)\eta^4(4\tau) \in S_4(\Gamma_0(8))$ from Theorem~\ref{T:main}.  Thus the $L$-function of $f$ appears in the $L$-function of $H_1$, which led us to hypothesize that $\m(R_{16})$ should be related to $L(f,4)$ in the first place ($t=1 \leftrightarrow k=16$).

For other values of $k$, work of McCarthy and the first author \cite{MP} shows that ${}_4F_3(-1)$ can be expressed in terms of eigenvalues of a Siegel eigenform of degree $2$, whose $L$-function is a tensor product $L$-function $L(f_2 \otimes f_3,s)$, for classical newforms $f_2 \in S_2(\Gamma_0(32))$ and $f_3 \in S_3(\Gamma_0(32),\chi_{-4})$ (see \cite{vGvS}).  Comparing with \eqref{FPR} this led us to investigate relations between
\[
  \m \bigl(R_{16\sqrt{-1}}\bigr) \longleftrightarrow L(f_2 \otimes f_3,4),
\]
but after several attempts by the authors to search for such a relationship, the question of finding one remains open.

On the other hand, let us consider the four-variable Laurent polynomial $x+x^{-1}+y+y^{-1}+z+z^{-1}+w+w^{-1}$, studied in \cite{AO1,AO2,VE}. By an elementary change of variables, we have
\begin{equation}\label{degenerate}
\m(x+x^{-1}+y+y^{-1}+z+z^{-1}+w+w^{-1})=R(1)=\frac{7\zeta(3)}{2\pi^2}.
\end{equation}
It would be quite interesting to explain why $\zeta(3)/\pi^2$ appears on the right, rather than a special value of the $L$-function attached to the three-fold, which also coincides with $L(f,s)$.  From examples involving elliptic curves, it seems plausible that there are some arithmetic conditions which the polynomial fails to satisfy, which are necessary to produce formulas such as \eqref{E:main}.  Notice that $\zeta(3)/\pi^2$ is essentially an analogue of $G/\pi$, which appears in connection with the genus zero curve in equation \eqref{E:2var}.  In summary, it would be interesting to fully explain \eqref{degenerate}, as it might help to predict what $L$-values should appear in additional Mahler measure formulas.

Other possible projects would be to study the Mahler measures of the following families:
\begin{align*}
S_k:&=x+x^{-1}+y+y^{-1}+z+z^{-1}+w+w^{-1}-k,\\
T_k:&=x^5+y^5+z^5+w^5+1-kxyzw.
\end{align*}
The latter can be viewed as a four-dimensional analogue of the Hesse family of elliptic curves $x^3+y^3+1-kxy$ and the family $x^4+y^4+z^4+1-kxyz$ of $K3$ surfaces, whose Mahler measures are known to be related to special $L$-values \cite{RV,Rogers,Samart2,Samart1}.

\textbf{Acknowledgements.}  The authors thank Wadim Zudilin for the useful suggestions which improved the exposition of the paper, and also for bringing Verrill's paper to our attention.

\bibliographystyle{amsplain}

\end{document}